\documentclass[11pt]{amsart}
\usepackage{amsmath}
\usepackage{graphicx}
\usepackage{epstopdf}
\usepackage{amssymb,amscd,array}
\usepackage{color}
\usepackage{float}

\makeindex \makeatletter
\def\captionof#1#2{{\def\@captype{#1}#2}}
\makeatother

\newcounter{tablegroup}
\newcounter{subtable}[tablegroup]

\newtheorem{thm}{Theorem}[section]
\newtheorem{cor}[thm]{Corollary}
\newtheorem{lem}[thm]{Lemma}
\newtheorem{prop}[thm]{Proposition}
\newtheorem{defn}[thm]{Definition}

\numberwithin{equation}{section}

\newcommand{\eps}{\varepsilon}

\begin{document}
\title[Minimal sets and orbit space]
{Minimal sets and orbit space for group actions on local dendrites}

\author{ Habib Marzougui and Issam Naghmouchi}

\address{ Habib Marzougui, University of Carthage, Faculty
of Sciences of Bizerte, Department of Mathematics,
Jarzouna, 7021, Tunisia.}
\email{hmarzoug@ictp.it and habib.marzougui@fsb.rnu.tn}
\address{ Issam Naghmouchi, University of Carthage, Faculty
of Sciences of Bizerte, Department of Mathematics,
Jarzouna, 7021, Tunisia.}
 \email{issam.nagh@gmail.com and issam.naghmouchi@fsb.rnu.tn}



\begin{abstract}
 We consider a group $G$ acting on a local dendrite $X$ (in particular on a graph). We give a
 full characterization of minimal sets of $G$ by showing that any minimal set $M$ of $G$ (whenever $X$ is different from a dendrite)
 is either a finite orbit, or a Cantor set, or a circle.
If $X$ is a graph different from a circle, such a minimal $M$ is a finite orbit.
These results extend those of the authors for group actions on
dendrites. On the other hand, we show that, for any group $G$ acting
on a local dendrite $X$ different from a circle, the following properties are equivalent:  (1) ($G, X$) is
pointwise almost periodic. (2) The orbit closure relation $R = \{(x, y)\in X\times X: y\in \overline{G(x)}\}$ is
 closed. (3) Every non-endpoint of $X$ is periodic. In addition, if $G$ is countable and $X$ is a local dendrite, then
 ($G, X$) is pointwise periodic if and only if the orbit space $X/G$
 is Hausdorff.
\end{abstract}
\maketitle

\section{\bf Introduction}
Let $X$ be a compact metric space with a metric $d$ and $G$ be a discrete group.
By an action of $G$ on $X$ we mean a continuous map $\varphi: G\times X\longrightarrow X$ satisfying
$\varphi(e,x)= x$ and $\varphi(g_{1}g_{2},x) = \varphi\left(g_{1}, \varphi(g_{2},x)\right)$ for all $x\in X$,
and all
$g_{1}, g_{2}\in G$, where $e$ is the identity of $G$. For convenience we often use $g(x)$ to denote $\varphi(g,x)$. Obviously for each $g\in G$, the map $g: X\longrightarrow X; x\longmapsto g(x)$ is a homeomorphism of $X$.
For any $x\in X$, the subset $G(x)= \{
g(x): g\in G\}$ is called the \textit{orbit} of $x$ under $G$. A subset $A$ of $X$ is called \textit{$G$-invariant} if
$g(A) = A$, for every $g\in G$. It is called \textit{a minimal set of $G$} if it is
non-empty, closed, $G$-invariant and minimal (in the sense of
inclusion) for these properties, this is equivalent to say that it is an orbit closure that contains no smaller one;
for example a single finite
orbit. When $X$ itself is a minimal set, then we say that the action of $G$ on $X$ is \textit{minimal}.
One of the objectives of the theory of dynamical systems has been to characterize the topological structure of
minimal sets. Clearly, the answer depends on $X$. Let us first recall that every
group action on a compact metric space admits a minimal set, as results from Zorn's lemma.
Among one-dimensional compact spaces, the characterization of minimal sets is well known on the compact interval;
these are finite orbits (see Lemma \ref{l31}). For the circle, there are three possibilities for minimal sets
(see Corollary \ref{c33}).

Recent interest in dynamics on dendrites and local dendrites is motivated by the fact that local dendrites are examples of Peano continua
with complex topology structures (e.g., \cite{Nadler},
pp. 165--187). For continuous maps on dendrites and local dendrites,
a full topological characterization of minimal sets was given by Balibrea et al. in \cite{Baln}.

For groups actions on graphs and dendrites, several results related to minimality,
sensitivity and existence of global fixed points
have been obtained by some authors (see e.g.,  \cite{Sh1}, \cite{Ma3}, \cite{Sh2}, \cite{Sh3}). For
rigidity results for actions on dendrites, see
\cite{mon}.
Existence of minimal group actions on dendrites can occur in the study of $3$-hyperbolic geometry (see \cite{Ra},
p. 601). These facts, among others, motivate us to explore the topological dynamic of minimal sets for group actions on
local dendrites. In \cite{MN1}, the authors studied minimal sets for group actions on dendrites.
This paper is, in part, a continuation of that work;
 we will study minimal sets for group actions on local dendrites. Graphs and dendrites
are particular cases of local dendrites.
 Our main results are a full characterization of minimal sets on graphs
and local dendrites different from a dendrite (Theorems \ref{tr1} and \ref{tr37}).
On the other hand, for a group $G$ acting
on a local dendrite $X$ different from a circle, we show (see Theorem \ref{t56}) that the following properties are equivalent: (1) ($G, X$) is
pointwise almost periodic. (2) The orbit closure relation $R = \{(x, y)\in X\times X: y\in \overline{G(x)}\}$ is
closed. (3) Every non-endpoint of $X$ is periodic.
\medskip

The plan of the paper is as follows. In Section 2, we give some definitions and preliminary properties
on graphs, dendrites and local dendrites which are useful for the rest of the paper.
Section 3 is devoted to minimal sets for groups actions on local dendrites. In particular we prove
that any minimal set for group actions on graphs different from a circle is a finite orbit.
In Section 4, we deal with  the family of minimal sets considered in the hyperspace of closed subsets endowed with the Hausdorff metric. Section 5 is devoted to the
relation between almost periodicity, closure orbit relation, and the orbit (class) space 
for groups actions on local dendrites.
\medskip

\section{\bf Preliminaries}
In this section, we recall some basic properties of graphs, dendrites and local dendrites.

A continuum is a compact connected metric space. An arc is
any space homeomorphic to the compact interval $[0,1]$. A topological space
is arcwise connected if any two of its points can be joined by an
arc. We use the terminologies from Nadler \cite{Nadler}.
\medskip

By a \textit{graph} $X$, we mean a continuum which can be written as the union of finitely many arcs
such that any two of them are either
disjoint or intersect only in one or both of their endpoints. Each of these arcs is called an \textit{edge} of the graph. A point $v\in G$
is called a \textit{branch point} if it admits a neighborhood $U$ in $X$
homeomorphic to the set $\{z\in \mathbb{C}: z^{r}\in [0,1]\}$ with
the natural topology for some integer $r\geq 3$, with the
homeomorphism mapping $v$ to $0$. If $r=1$, then we call $v$ an
\textit{endpoint} of $X$. Denote by $B(X)$ and $E(X)$ the sets of
branch points and endpoints of $X$ respectively. An edge is the
closure of some connected component of $X\setminus B(X)$, it is
homeomorphic to $[0,1]$. A subgraph of $X$ is a subset of $X$ which
is a graph itself. Every sub-continuum of a graph is a graph
(\cite{Nadler}, Corollary 9.10.1).
 Denote by $S^{1}=[0,1]_{\mid 0\sim 1}$ the unit circle endowed with the
orientation: the counter clockwise sense induced via the natural
projection $[0,1]\rightarrow S^{1}$. A circle is any space homeomorphic to $S^{1}$.
\medskip

By a \textit{dendrite} $D$, we mean a locally connected continuum
containing no homeomorphic copy to a circle. Every sub-continuum of a
dendrite is a dendrite (\cite{Nadler}, Theorem 10.10) and every
connected subset of $D$ is arcwise connected (\cite{Nadler},
Proposition 10.9). In addition, any two distinct points $x,y$ of a
dendrite $D$ can be joined by a unique arc with endpoints $x$ and
$y$, denote this arc by $[x,y]$ and let denote by
$[x,y)=[x,y]\setminus\{y\}$ (resp. $(x,y]=[x,y]\setminus\{x\}$ and
$(x,y)=[x,y]\setminus\{x,y\}$). A point $x\in D$ is called an
\textit{endpoint} if $D\setminus\{x\}$ is connected. It is called a
\textit{branch point} if $D\setminus \{x\}$ has more than two
connected components. The number of connected components of $D\setminus \{x\}$ is called the \textit{order} of $x$.
Denote by $E(D)$ and $B(D)$ the sets of
endpoints, and branch points of $D$, respectively. A point $x\in
D\setminus E(D)$ is called a \textit{cut point}. The set of cut
points of $D$ is dense in $D$.
Following (\cite{Ar}, Corollary 3.6), for any dendrite $D$, we have
B($D)$ is discrete whenever E($D)$ is closed. 
For a subset $A$ of $D$, we call \emph{the convex hull} of $A$, denoted by $[A]$, the intersection of all 
sub-continuums of $D$ containing $A$. If $A$ is a sub-dendrite of $D$, define the retraction $r_{A} : D \rightarrow A$
by letting $r_{A}(x) = x$, if $x\in A$,
and by letting $r_{A}(x)$ to be the unique point $r_{A}(x)\in A$ such that $r_{A}(x)$ is a point of any arc in $D$
from $x$ to any point of $A$, if $x\notin A$ (see \cite{Nadler}, 10.26, p. 176). Note that the map $r_{A}$ is monotone and it is constant on each
connected component of $D\backslash A$.
\medskip

By a \emph{local dendrite} $X$ we mean a continuum having the property that every of its points has a neighborhood
which is a dendrite. A local dendrite is then a locally connected continuum containing only a finite number of circles
(\cite{Kur}, Theorem 4). As a consequence every sub-continuum of a local dendrite is a local dendrite (\cite{Kur}).
Every graph and every dendrite is a local dendrite.
Let $X$ be a local dendrite. For any arc $I$ in $X$, we denote by $\gamma(I)$ the set of its endpoints.
A point $x\in X$ is called a \emph{branch point} of $X$ if there exists a closed neighborhood $D$
of $x$ which is a dendrite such that $x$ is a branch point of $D$ (i.e. $D\backslash \{x\}$ has more than two connected
components). We denoted by $B(X)$ the set of branch points of $X$.
By (\cite{Kur}, Theorem 6, 304 and Theorem 7, 302), $B(X)$ is at most countable.
\medskip

For a subset $A$ of $X$, denote by $\overline{A}$ the closure of $A$ and by $\textrm{diam}(A)$ the diameter of
$A$.
\medskip

\begin{lem}[\cite{IG}, Lemma 2.3]\label{l2} Let $X$ be a local dendrite,
$(C_{i})_{i\in\mathbb{N}}$ be a sequence of pairwise disjoint connected subsets of $X$.
Then $\underset{n\to +\infty}\lim \mathrm{diam}(C_{n})=0$.
\end{lem}

\begin{lem}\label{oc} Let $X$ be a dendrite and let $U$ and $V$ be two disjoint non-empty connected subsets of $X$.
Then $\overline{U}\cap \overline{V}$ contains at most one point.
\end{lem}

\begin{lem}[\cite{am}, Lemma 4.3]\label{diam} Let $X$ be a local dendrite with metric $d$. Then for any $\eps>0$ there is $0<\delta<\eps$ such that if $d(x,y)<\delta$ then
$\mathrm{diam}([x,y])<\eps$.
\end{lem}

\begin{lem}\cite{Nagh2}\label{l34} Let $D$ be a dendrite with countable set of endpoints. Then every sub-dendrite of $D$ has countable set of endpoints.
\end{lem}

\begin{lem}\cite{Ar}\label{l35} Let $D$ be a dendrite with closed set of endpoints. Then we
have:
\begin{itemize}
 \item[(i)] $\overline{B(D)}\subset B(D)\cup E(D)$.

 \item[(ii)] Every sub-dendrite of $D$ has a closed set of endpoints.
\end{itemize}
\end{lem}
\medskip

From (\cite{Ar}, Theorem 3.3 ), we deduce easily the following Lemma:
\medskip

\begin{lem}\label{l12} The order of every branch point of a dendrite with closed set of endpoints is finite.
\end{lem}
\medskip

The following Lemma is trivial (see also \cite{Nagh1}).

\begin{lem}\label{g1} Let $X$ be a local dendrite and $f: X\rightarrow X$ a homeomorphism. Then:
\begin{itemize}
 \item[(i)] $f(B(X))=B(X)$.
 \item[(ii)]  $f(E(X))=E(X)$.
\end{itemize}
\end{lem}

\begin{lem}\label{l21} Let $X$ be a continuum, a group $G$ acting on $X$ and $M$ a minimal set of $G$. Then $M$ is
\begin{itemize}
 \item[(i)] a finite orbit, or
  \item[(ii)] $X$; in this case all orbits are dense, or
\item[(iii)] a $G$-invariant, compact perfect nowhere dense subset of $X$.
\end{itemize}
\end{lem}

\begin{proof} If $M$ contains a point $a$ isolated in $M$, then so is $g(a)$ for any $g\in G$. Hence
$\overline{G(a)}\backslash G(a)\subset M$ is closed, $G$-invariant and then, by minimality of $M$, we have
$\overline{G(a)} = G(a)$. Thus $G(a)$ is a finite orbit and $M = G(a)$.
So assume that $M$ has no isolated point, i.e. $M$ is perfect. If $M$ is somewhere dense in $X$, i.e. its interior $\mathring{M}\neq \emptyset$, then
$M\backslash \mathring{M}$ is closed and $G$-invariant, then by minimality of $M$, we have $M =  \mathring{M}$ and hence $M= X$. This completes the proof.
 \end{proof}
\medskip
\medskip

\section{\bf Minimal sets on local dendrites}
\medskip

\textbf{3.1. The interval case} \begin{lem}[Interval case]\cite{MN1}\label{l31}
Let $G$ be a group acting on the closed interval $I$ and $M\subset I$ a minimal set of $G$. Then
$M$ is a finite orbit (in fact a single point or two points).
\end{lem}
\medskip

\textbf{3.2. The circle case.} Let $\textrm{Homeo}(S^{1})$ (resp. $\textrm{Homeo}^{+}(S^{1})$)
be the group of homeomorphisms (resp. orientation preserving homeomorphisms) of $S^{1}$.
\medskip

\begin{prop}[Circle case]\label{p32} \cite{Bek}
Let $G$ be a subgroup of $\mathrm{Homeo}^{+}(S^{1})$ and $M\subset S^{1}$ a minimal set of $G$. Then $M$ is
\begin{itemize}
 \item[(i)] a finite orbit, or
  \item[(ii)] $S^{1}$; in this case all orbits are dense, or
\item[(iii)] a Cantor set; in this case it is contained in the closure of any orbit and hence unique.
\end{itemize}
  \end{prop}
  \medskip

The following result is due to \cite{Bek}. We present its proof for completeness.

\begin{cor}\label{c33}
Let $G$ be a subgroup of $\mathrm{Homeo}(S^{1})$ and $M\subset S^{1}$ a minimal set of $G$. Then $M$ is
\begin{itemize}
 \item[(i)] a union of at most two finite $G^{+}$-orbits, (where $G^{+} = G\cap \textrm{Homeo}^{+}(S^{1})$), or
  \item[(ii)] $S^{1}$; in this case all orbits are dense,

  or
\item[(iii)] a Cantor set; in this case it is contained in the closure of any $G$-orbit and hence unique.
\end{itemize}
  \end{cor}

  \begin{proof} As $M$ is closed and $G^{+}$-invariant, there exists a minimal set $A\subset M$ of $G^{+}$.
  If $A=S^{1}$ then $M = S^{1}$. So assume that $M\neq S^{1}$ and that $G\backslash G^{+}\neq \emptyset$.
 So let $h\in G\backslash G^{+}$. Then $h(A)$ is $G^{+}$-invariant (since for any $g\in G^{+}$, we have
 $h^{-1}gh\in G^{+}$ and hence $h^{-1}gh(A)=A$), moreover it is minimal for $G^{+}$.
 If $A$ is a Cantor set then by Proposition \ref{p32}, $h(A)= A$. Hence $A$ is $G$-invariant
 and therefore $A=M$. We conclude that $A$ is contained in the closure of any $G$-orbit and hence unique.
 If $A$ is a finite $G^{+}$-orbit, then so is $h(A)$.
 As $A\cup h(A)\subset M$ is $G$-invariant; indeed, for any $g\in G\backslash G^{+}$, we have $g(h(A))= A$ and
 $g(A) = gh(h(A)) = h(A)$ since
 $h^{2}(A)= A$ and $gh\in G^{+}$, hence $g(A\cup h(A)) = A\cup h(A)$. We conclude that $M = A\cup h(A)$.
 This completes the proof.
\end{proof}
\medskip

In particular:

\begin{cor}\label{c100} Let $G$ be a subgroup of $\mathrm{Homeo}(S^{1})$ and $M\subset S^{1}$ a minimal set of $G$.
Assume that $G$ has a finite orbit. Then $M$ is a finite orbit.
\end{cor}

%

\begin{prop}\label{p35} Let $G$ be a subgroup of $\mathrm{Homeo}(S^{1})$. Assume that $G$ has a finite orbit. Then:

\begin{item}
 \item[(i)] All finite orbits have the same cardinal $p$ if $G$ is a subgroup of $\mathrm{Homeo}^{+}(S^{1})$.
 \item[(ii)] All finite orbits have cardinal $p$ or $2p$ if $G$ is a subgroup of $\mathrm{Homeo}(S^{1})$.
\end{item}
\end{prop}
\medskip

\begin{proof} Let $O$ be a finite orbit of $G$ of cardinal $p$,
set $O = \{a_1, \dots, a_p=a_1\}$ in the natural order on $S^{1}$.

(i) Assume first that $G$ be a subgroup of $\mathrm{Homeo}^{+}(S^{1})$.
Let $G(x)$ be a finite orbit of $G$ through $x$ and denote by
$I_i = [a_i, a_{i+1}]\subset S^{1}$. We show that $G(x)\cap I_i\neq \emptyset$ for any $i$: Indeed, we have
$x\in I_{i_{0}}$ for some $i_0\in \{1, \dots, p\}$ and
$g(a_{i_{0}}) = a_{i}$ for some $g\in G$. Then $g(I_{i_{0}})= I_{i}$ and so $G(x)\cap I_i\neq \emptyset$.
Now let us show that $G(x)\cap I_i$ is reduced to a point: Let $y$ be the nearest
point of $G(x)\cap I_i$ to the point $a_i$.
Suppose that $G(x)\cap I_i$ contains other point than $y$. In this case, there is $g\in G$ such that
$g(y)\in I_i$ with $y\in (g(y), a_i)$. As $g\in \mathrm{Homeo}^{+}(S^{1})$ and
$g(I_{i}) = I_{i}$, so the set $\{g^{n}(y):n\geq 1\}$ is infinite and included in $G(x)$,
a contradiction. We conclude that $G(x)$ has cardinal $p$.

(ii) Now assume that $G\backslash G^{+}\neq \emptyset$, where $G^{+} = G\cap \textrm{Homeo}^{+}(S^{1})$. Let $G(x)$ be a finite orbit of $G$ through $x$.
Then $G(x) = O_1 \cup h(O_1)$, where $O_1$ is a finite $G^{+}$-orbit. If $h(O_1) = O_1$, then
card($G(x)) =p$, where $p$ is the cardinal of any finite $G^{+}$-orbit, by above. If $h(O_1)\neq O_1$, then card($G(x)) =2p$.
This completes the proof.
\end{proof}
\medskip

\textbf{Remark 1.} The property (ii) in Proposition \ref{p35} can occur; for example,
take $G =\{\textrm{id}, \sigma\}$, where
$\sigma(z)= \overline{z}:$ the conjugate of $z\in S^{1}$. In this case, $G$ has two fixed points
(i.e. of cardinal $1$) and all the other orbits are finite of cardinal $2$.
\medskip

\textbf{3.3. The graph case.} We have the following theorem:
\medskip

\begin{thm}\label{tr1} Let $X$ be a graph different from a circle and a group $G$ acting on $X$. Then a minimal set $M$
of $G$ is finite (in fact a finite orbit).
\end{thm}

\begin{proof}
  The case where $X$ is a closed interval is already down (see Lemma \ref{l31}). So assume
that $X$ is not an interval. First we have $M\neq X$ since, by Lemma \ref{g1}, for any $b\in B(X)$, $G(b)\subset B(X)$,
hence $G(b)$ is finite. If $M\cap (B(X)\cup E(X))\neq \emptyset$, there is a point $v\in M\cap B(X)$
(resp. $e\in M\cap E(X)$), then
 $G(v) = M$ (resp.  $G(e)= M$) is finite by the minimality of $M$ and Lemma \ref{g1}. Assume that $M\cap (B(X)\cup E(X))=\emptyset$. Since $X$ is not an interval and not a circle, there is $v\in B(X)$. Let $U$ be the connected component of
$X\backslash M$ containing $v$. So there exist a point $a\in M$ and an arc $I$ with $\gamma(I)=\{a,v\}$ and such that
$U\cap I=I\setminus\{a\}$. For each $g\in G$, $g(I)$ is an arc with $\gamma(g(I))=\{g(v),g(a)\}$ and such that
$g(I)\cap g(U)= g(I)\setminus\{g(a)\}$. Because $X$ is a graph and the orbit $Gv$ is finite, the family of arcs
$\{g(I): \ g\in G\}$ is finite. Hence the orbit $Ga$ is finite and thus
$M = G(a)$ is a finite orbit.
  \end{proof}
\medskip


\textbf{3.4. The dendrite case.} We recall following results:
\medskip

\begin{thm}[Dendrite case] $($\cite{MN1}, Theorem 3.1$)$\label{tr33}
Let $X$ be a dendrite, a group $G$ acting on $X$ and $M$ a
minimal set of $G$. Then $M$ is
\begin{itemize}
 \item[(i)] a finite orbit, if $E(X)$ is countable.

\item[(ii)] a finite orbit, or a Cantor set included into $E(X)$, if $E(X)$ is closed.
\end{itemize}
\end{thm}
\medskip

\begin{prop} $($\cite{MN1}, Corollary 5.5$)$\label{pr7} Under the hypothesis of Theorem \ref{tr33},
if the action has a finite orbit, then $M$ is either a finite orbit, or a Cantor set.
\end{prop}
\medskip

\textbf{3.5. The local dendrite case.} We have the following theorem:
\medskip

\begin{thm}\label{tr37} Let a group $G$ acting on a local dendrite $X$ different from a dendrite and $M$ a
minimal set of $G$. Then $M$ is either a finite orbit, or a Cantor set or a circle.
\end{thm}

Before the proof of Theorem \ref{tr37}, we introduce the invariant graph. Let $X$ be a local dendrite different from a dendrite.
We denote by $Y$ the minimal graph (in the sense of inclusion) which contain all the circles in $X$ (i.e.
the intersection of all graphs in $X$ that contain all the circles in $X$). By (\cite{Ah}, Proposition 3.6), $Y$ is $G$-invariant.

We define the quotient space: Fix a point $\gamma$ in $Y$ and collapse the graph $Y$ to the point $\gamma$,
we obtain the quotient space
$\widehat{X} = (X\backslash Y)\cup \{\gamma\}$. Let $\pi: X\to \widehat{X}$ be the quotient map defined by:
\begin{center}
    $\pi(x)=\left\{
                    \begin{array}{ll}
                      \gamma, \ & \textrm{if} \ x\in Y \\
                       x, \ & \textrm{if} \ x\in X\backslash Y \\
                    \end{array}
                    \right.$
                    \end{center}
\medskip

     We endow $\widehat{X}$ by the metric $\widehat{d}$ defined as follows:
\begin{center}
    $\widehat{d}(\pi (x),\pi (y)) = \left\{
                    \begin{array}{ll}
                      d(x,y), \ & \textrm{if} \ x, y\in X\backslash Y \\
                      \max\left(d(x,Y), d(y,Y)\right), \ &  \textrm{otherwise}\\
                    \end{array}
                    \right.$
                    \end{center}

Since $\widehat{d}(\pi (x),\pi (y))\leq d(x,y)$, for any $x, y\in X$, the map $\pi$ is continuous.
Therefore $\pi$ is closed and onto, which implies that $\widehat{X}$ is a continuum. As a consequence:

\begin{lem} $(\widehat{X},\widehat{d})$ is a dendrite.
\end{lem}
\medskip

We define the new generated group $\widehat{G}: =\{\widehat{g}: g\in G\}$, where
$\widehat{g}: \widehat{X} \longrightarrow \widehat{X}$ is defined as follows:
$\widehat{g}(\pi(x)) = \pi(g(x)), \textrm{ for any } x\in X$. It is easy to see that $\widehat{G}$ is a
group acting by homeomorphisms on $\widehat{X}$.

Set $X\backslash Y = \bigcup_{i\in \mathcal{A}}C_{i}$,
where the $C_{i}$ are the connected
components of $X\backslash Y$ and $\mathcal{A}$ is at most countable. By (\cite{Ah}, Lemma 2.11), for
any $i\in \mathcal{A}$, $\overline{C_{i}}\cap Y$ is
reduced to a point. Let $\mathcal{A}^{\prime}\subset \mathcal{A}$.
For each $k\in \mathcal{A^{\prime}}\subset\mathcal{A}$, we define the set $C^{k}$ of $X$ and the point $z_{k}$ of $Y$ by
$C^{k}=\underset{i\in \mathcal{A^{\prime}}_{k}}\cup\overline{C_{i}}$, where $C_{i}$ are the connected components of
$X\backslash Y$
and $\mathcal{A^{\prime}}_{k}=\{i\in\mathcal{A}: \overline{C_{i}}\cap Y = \{z_{k}\}\}$.
\bigskip

\textit{Proof of Theorem \ref{tr37}}. If $M\cap Y\neq\emptyset$, then $M\subset Y$  is either a finite orbit, or a Cantor set or a circle and in this
later case, $Y$ is a circle (Theorem \ref{tr1}).
If $M\cap Y=\emptyset$, then $M\subset X\setminus Y$. By collapsing the graph $Y$ to a point $\gamma$, we get a
new dendrite $Z$ and a new generated action on $Z$ having $\gamma$ as a global fixed point.
Now apply Proposition \ref{pr7}
to obtain that $M$ is either a finite orbit, or a Cantor set. \qed

\begin{prop}\label{pr88} Let $X$ be a local dendrite, a group $G$ acting on $X$ and $M$ a minimal set of $G$.
Assume that $G$ has a finite orbit. Then $M$ is either a finite orbit, or a Cantor set.
\end{prop}

The proof is a consequence of two lemmas.

\begin{lem} $($\cite{Ah}, Lemma 3.8$)$ \label{f} Under the notation above, let $f: X\to X$ be an onto monotone
local dendrite map.
Let $x\in C^{k}$ for some $k\in\mathcal{A^{\prime}}$ and let $f^{n}(x)\in C^{i}$ for some
$i\in\mathcal{A^{\prime}}$ and $n\in\mathbb{N}$.
Then $f^{n}(z_{k})=z_{i}$.
\end{lem}
\medskip

\begin{lem}\label{bn} Let $X$ be a local dendrite, a group $G$ acting on $X$ and $M$ a
minimal set of $G$. Assume that $G$ has a finite orbit. Then $G$ has a finite orbit in $Y$.
\end{lem}

\begin{proof} Let $x\in X$ with $G(x)$ finite. If $x\in Y$ then $G(x)\subset Y$. Otherwise, there exists
$k\in\mathcal{A^{\prime}}$ such that $x\in C^{k}$. Set $G(x) = \{x, g_{1}x, \dots, g_{p}x\}$ and let $g\in G$.
So there exists
$i\in \{1, \dots, p\}$ such that $gx = g_{i}x$ and thus $g_{i}^{-1}gx=x$. By Lemma \ref{f}, $g_{i}^{-1}g(z_{k}) = z_{k}$, where $C^{k}\cap Y = \{z_{k}\}$
and so $G(z_{k})$ is a finite orbit in $Y$.
\end{proof}

\begin{proof}[Proof of Proposition \ref{pr88}] Assume that $G$ has a finite orbit. Then by Lemma \ref{bn},
$G$ has a finite orbit in $Y$. If $M\cap Y\neq \emptyset$, then $M\subset Y$ (since $M$ is a minimal set of $G$).
Hence $M$ is a finite orbit; this follows from Theorem \ref{tr1} if $Y$ is different from a circle and from Corollary \ref{c100} if $Y$ is a circle. If $M\cap Y = \emptyset$, then by Theorem \ref{tr37}, $M$ is either a finite orbit, or a Cantor set.
\end{proof}
\medskip

\section{\bf Minimal sets in the hyperspace}

Given a continuum $X$ with a metric $d$, we denote by $2^{X}$ the hyperspace
of all nonempty closed subsets of $X$. For any two subsets $A$ and $B$ of
$X$, we denote by $d(A,B) = \inf_{x\in A, y\in B}d(x,y)$ and $d(x,A) = d(\{x\},A)$.
The Hausdorff metric $d_{H}$ on $2^{X}$ is defined as follows: for $A, B\in 2^{X}$,
\begin{center}
$d_{H}(A,B) = \max\{\sup_{x\in A}d(x,B), \ \sup_{y\in B}d(y, A)\}$.
\end{center}

This defines a distance on $2^{X}$ (\cite{Nadler}, Theorem 4.2). With this distance,
$2^{X}$ is a compact metric space (\cite{Nadler}, Theorem 4.3).
If $X$ is a continuum and $f: X\to X$ is a map we can consider the map  $2^{f}: 2^{X}\to 2^{X}$  (called the induced map) defined as follows:
$2^{f}(A)=f(A)$, for each $A\in 2^{X}$. If $f$ is continuous, then $2^{f}$ is also continuous (\cite{Illanes99}, Lemma 13.3).
\bigskip

The aim of this section is to prove the following theorem which generalizes the author's theorem
(\cite{MN1}, Theorem 6.9) to local dendrites.
\medskip

\begin{thm}\label{t610} Let $G$ be a group acting on a local dendrite $X$. Then the set of all minimal sets of $G$
endowed with the Hausdorff metric is compact. This holds if $X$ is a graph.
\end{thm}
\medskip

\textit{Proof.} Let $(M_{n})_{n\in \mathbb{N}}$ be a sequence of minimal sets of $G$ converging
in the Hausdorff metric to a closed subset $M$ of $X$. Then $M$ is closed and $G$-invariant.
We distinguish three cases.
\bigskip

 \textbf{Case 1: $X$ is a circle.} If $G$ has no finite orbit, then it has a unique minimal set
(Corollary \ref{c33}) and so the theorem follows.
If $G$ has a finite orbit, then by Corollary \ref{c100}, each $M_{n}$ is a finite orbit.

Assume that $G$ be a subgroup of $\mathrm{Homeo}^{+}(S^{1})$.
Then by Proposition \ref{p35},
 all finite $G$-orbits have the same cardinal, say $p$. So card($M_{n}) = p$ for any $n$.
Therefore the limit set $M$ is finite of cardinality at most $p$; this follows from the fact that
the hyperspace of finite subsets of $X$ of cardinality
at most $p$ is a continuum (see \cite{bu}).

Since $M$ is $G$-invariant, so it is a single finite orbit (otherwise, it contains a finite orbit with cardinal less than $p$, a contradiction). We conclude that $M$ is a finite orbit which is a minimal set of $G$.

Now assume that $G\backslash G^{+}\neq \emptyset$, where $G^{+} = G\cap \textrm{Homeo}^{+}(S^{1})$. Let $h\in G\backslash G^{+}$.
By the proof of Corollary \ref{c33}, $M_n= O_n \cup h(O_n)$, where $O_n$ is a finite $G^{+}$-orbit. If
$h(O_n)= O_n$, for infinitely many $n$, then by above, $M$ is a
finite $G^{+}$-orbit. One can suppose that $h(O_n) \neq O_n$, for every $n$. Then card($M_n)= 2p$ and so
card($M)=2p$ and $M = O\cup h(O)$, where $O$ is a finite $G^{+}$-orbit and hence $M$ is a minimal
set for $G$.
%
%
\bigskip

\textbf{Case 2: $X$ is a graph different from a circle.}
Denote by


- $\mathcal{C} = \{C_{1},C_{2},\dots,C_{t}\}$ the set of circles in
$X$ containing only one vertex.

\medskip

Denote by $\mathcal{C} = \{C_{1},C_{2},\dots,C_{k}\}$ the set of all the circles in
$X$.
\medskip


\begin{lem}\label{l49} If $X$ is a graph different from a circle and $f: X\rightarrow X$ is a homeomorphism,
then for every $C\in \mathcal{C}$, $f(C)\in \mathcal{C}$.
\end{lem}
\medskip

\begin{lem}\label{gr1}
Let $I$ be an edge in $X$ and let $\textrm{Sat}(I): = \underset{g\in G}\cup g(I)$. Let $p$ be
the number of elements in the family $\{g(I): \ g\in G\} = \{g_0(I),g_1(I),\dots,g_{p-1}(I)\}$,
where $g_{0} = \textrm{id}_X$. Then any finite orbit $O$ meeting the interior int$(I): = I\backslash \gamma(I)$  has cardinal $p$ or $2p$.
Moreover, if there is an orbit $O$ with cardinal $2p$, then there are $p$ elements $h_1,\dots, h_p$ of $G$ such
that for any orbit $M$ with cardinal $2p$, $M\cap g_i(I) = \{x_0(i),x_1(i)\}$, where $x_0(i)\neq x_1(i)$
and $h_i(x_k(i))=x_{k+1 \mod(2)}(i)$, $k=0,1$.
\end{lem}
\medskip

\begin{proof}
Let $O$ be a finite orbit with non-empty intersection with int$(I)$. Thus $O\subset
\textrm{int}(\textrm{Sat}(I))$. First observe that the sets $O_i: = O\cap g_i(I)$, for $i=0,\dots,p-1$
have the same number
of points. Moreover, for each $0\leq i \leq p-1$, $O_i:= O\cap I_i$ is in fact a finite orbit
for the action of the subgroup $G_i:= \{g\in G: \ g(I_i) = I_i\}$ on $I_i$. By Lemma \ref{l31},
$O_i$ is a single point or two points. Therefore, the orbit has cardinal $p$ or $2p$.

Now, suppose that there is a finite orbit $O$ in $\textrm{Sat}(I)$ with $2p$ points. If
$O\cap \gamma(I)\neq\emptyset$, then $O\subset \underset{0\leq i \leq p-1}\cup \gamma(g_i(I))$
and in this case the arcs $g_i(I)$ are pairwise disjoint.
Hence for each $i$, $O\cap g_i(I)=\{x_i,y_i\}$ is an orbit for the action of $G_i$ on $g_i(I)$.
The same holds if $O\cap \gamma(I) = \emptyset$, in this case, $O\cap g_i(I) = \{x_i,y_i\}$
is an orbit for the action of $G_i$ on $g_i(I)$. In both cases, for each $i$, there is $h_i\in G_i$
such that $h_i(x_i)=y_i$ and $h_i(y_i)=x_i$. Hence if we denote by $\{a_i,b_i\} = \gamma(I_i)$,
then $h_i(a_i) = b_i$ and $h_i(b_i) = a_i$. So there is a fixed point $c_i\in \textrm{int}([x_i, y_i])\subset g_i(I)$.
So take any orbit $N$ in $Sat(I)$ with $2p$ points. It is easy to check that for each $i$,
$N\cap g_i(I) = \{x_0(i),\ x_1(i)\}$, where $x_0(i)\neq x_1(i)$ and $h_i(x_k(i)) = x_{k+1 \mod(2)}(i)$,
$k=0,1$.
\end{proof}

When replacing the edge $I$ by a circle $C$
containing only one vertex, the following lemma is similar to the lemma above.

\begin{lem}\label{gr2}
Let $C$ be a circle in $X$ containing only one vertex and let $\textrm{Sat}(C) =
\underset{g\in G}\cup g(C)$. Let $p$ be the number of elements in the family
$\{g(C): \ g\in G\} = \{g_0(C),g_1(C),\dots,g_{p-1}(C)\}$, where $g_0 = \mathrm{id}_X$.
Then any finite orbit $O$ with non-empty intersection with $\mathrm{int}(C)$ has cardinal $p$ or $2p$.
Moreover, if there is an orbit with cardinal $2p$, then there are $p$ elements $h_1,\dots, h_p$
of $G$ such that for any orbit $M$ with cardinal $2p$, $M\cap g_i(I) = \{x_0(i),\ x_1(i)\}$,
where $x_0(i)\neq x_1(i)$ and $h_i(x_k(i)) = x_{k+1 \mod(2)}(i)$, $k=0,1$.
\end{lem}

\begin{proof}
Similar steps as in Lemma \ref{gr1} will be repeated to prove the first part of the Lemma. So let us prove the last part which is slightly different.

Suppose that $M$ is an orbit with $2p$ points in $\overline{C}$. Then $M$ intersects each $g_i(C)$ exactly in two 
points $x_i,y_i$ which form an orbit for the subgroup $G_i$ defined similarly as in Lemma \ref{gr1}. In the circle 
$C_i$, there is two arcs $A_i$ and $B_i$ joining $x_i$ and $y_i$. Suppose that $A_i$ contains the only vertex in $C_i$, 
namely $c_i$. Obviously, $c_i$ is fixed by any element in $G_i$. So take $h_i$ any element in $G_i$ such that 
$h_i(x_i)=y_i$, hence $h_i(A_i)=A_i$ and so $h_i(y_i)=x_i$. Also, we get that $h_i(B_i)=B_i$ and there is a point $w_i$
in the interior of $B_i$ fixed by $h_i$. Now take any orbit $N$ in $\overline{C}$ with $2p$ points. 
Then for each $i$, $N$ intersects $g_i(C)$ exactly in two points which both belong either to $A_i$ or $B_i$ 
and in either cases it is easy to check that they are permuted by $h_i$.
\end{proof}
\medskip

We continuous the proof of Theorem \ref{t610} as follows:
\medskip

If $M_n$ intersects $B(X)\cup E(X)$ for infinitely many $n$, then for infinitely many $n$, $M_n$
is the same orbit $O$ hence the limit of $(M_n)_n$ is $O$ which is a finite orbit.

If $M_n \cap B(X)\cup E(X) = \emptyset$ for $n$ large enough, the we can assume without loss of
generality that $M_n \cap B(X)\cup E(X) = \emptyset$ for each $n$. We distinguish two subcases:

\textit{Case 1.} For infinitely many, $M_n$ intersects the interior of an edge $I$:
Again without loss of generality, we can assume that this assertion holds for each $n$.

In this case, by applying Lemma \ref{gr1}, there is a number $p$ such that for each $n$, $M_n$ has
$p$ or $2p$ elements. If for a subsequence, $\textrm{card}(M_{\phi(n)})=p$, then
let $g_0,\dots,g_{p-1}$ be defined as in Lemma \ref{gr1}, and so for each $n$,
$M_{\phi(n)} = \{g_0(x^n),\dots,g_{p-1}(x^n)\}$. One can assume that $x^n$ converges to $x$.
Thus, $M:=\{x,g_1(x),\dots, g_{p-1}(x)\}$ is the limit in the sense of Hausdorff of
$(M_{\phi(n)})_n$. As for each $n$, $M_n$ is minimal, then its Hausdorff limit is closed and
$G$-invariant. Thus $M:=\{x,g_1(x),\dots, g_{p-1}(x)\}$ is closed and $G$-invariant and
hence it is a finite orbit.

If now for $n$ large enough, $\textrm{card}(M_{n})=2p$, from Lemma \ref{gr1}, there are
$g_0,\dots, g_{p-1}, h_0,\dots,h_{p-1}\in G$ such that for $n$ large enough, the orbit $M_n$
can be written as follows: $$M_n = \{x^n,g_1(x^n),\dots,g_{p-1}(x^n)\}\cup\{y^n,g_1(y^n),\dots,
g_{p-1}(y^n)\},$$ where for each $i$, $h_i(g_i(x^n)) = g_i(y^n)$ and $h_i(g_i(y^n)) = g_i(x^n)$.
By taking a suitable subsequence, we can assume that $x^n$ and $y^n$ converge to $x$ and $y$,
respectively in $X$. Similarly as above, we prove that $M$ (which is the Hausdorff limit of $M_n$)
is a finite orbit.
\medskip

\textit{Case 2.} For infinitely many $n$, $M_n$ intersects a circle $C$ which contains only one vertex. In this case, we will use Lemma \ref{gr2} and follow the same steps as in Case 1.
\medskip

\textbf{Case 3: $X$ is a local dendrite.} We distinguish two cases:
\medskip

\textit{\it (c-1)}: $M_{n}\cap Y\neq\emptyset$ for infinitely many integer $n$.
Without loss of generality, we assume that $M_{n}\cap Y\neq\emptyset$ for any $n\in \mathbb{N}$.
Since $M_{n}$ is a minimal set of $G$, so $M_{n}\subset Y$ for any $n\in \mathbb{N}$ and hence
$M\subset Y$. By the case 2, $M$ is a minimal set of $G$.
\medskip

\textit{\it (c-2)}: $M_{n}\cap Y = \emptyset$ for $n$ large enough. Without loss of generality,
one can assume that $M_{n}\cap Y=\emptyset$, for any $n\in \mathbb{N}$.
By collapsing $Y$ to a point $\gamma$, $\widehat{X}$ is a tree and $M_{n}$ is a minimal set in $\widehat{X}$
for $\widehat{G}$ for any $n\in \mathbb{N}$; indeed, if $x\in M_{n}$, then $\overline{G(x)} = M_{n}$ and thus $G(x)\cap Y = \emptyset$.
Hence $\pi(g(x)) = g(x)$ for all $g\in G$ and so $\overline{\widehat{G}(x)} = M_{n}$.
By the case 2, $M$
is a minimal set of $\widehat{G}$. As $\pi$ is continuous,
then $\pi(M)=M$. Let us show that $M$ is minimal set of $G$:

- If $M\cap Y\neq\emptyset$ then $\gamma\in M$ since
$\pi(M\cap Y) = \{\gamma\}\subset \pi(M)=M$. As $M$ is a minimal set of $\widehat{G}$ and $\gamma$ is a
fixed point of $\widehat{G}$ (since $g(\{\gamma\}) = g(\pi(Y)) = \pi(g(Y)) = \pi(Y) = \{\gamma\}$
for all $g\in G$), then $M = \{\gamma\}$. So $\gamma$ is a fixed point of $G$ (since $g(M) = M)$,
for all $g\in G$) and hence $M$ is a minimal set of $G$.

- If $M\cap Y = \emptyset$ then for any $x\in M$, $\overline{\widehat{G}(x)} = \overline{G}(x) = M$.
This proves that $M$ is a minimal
set of $G$. \qed
\medskip

\begin{cor}\label{c6100} Let $G$ be a group acting on a local dendrite $X$. Then the union of all minimal sets of
$G$ is closed in $X$.
\end{cor}
\medskip

\section{\bf Almost periodicity, orbit closure relation and orbit space}

Let $X$ be a compact metric space with metric $d$ and an action of a group $G$ on $X$. First, we recall a basic notion needed later.
A subset $S$ of $G$ is called \textit{syndetic} if there exists a finite subset $F$ of $G$ such that for each element
 $g$ in $G$, there is some element $k\in F$ with $kg\in S$; that is $G = FS$.

\begin{defn} \label{d51} \rm{A point $x\in X$ is said to be almost periodic for $G$ if for any neighborhood $U$ of $x$, the set of
return times
$N(x, U) = \{g\in G: gx\in U\}$ is syndetic in $G$.
We say that:

($G, X$) is \textit{pointwise almost periodic} if every point of $X$ is almost periodic.

($G, X$) is \textit{pointwise periodic} if every point of $X$ has finite orbit.

}
\end{defn}

We have the following characterization of almost periodicity via minimality:

\begin{prop}\label{p51n} $($\cite{go}$)$. Let $X$ be a compact metric space. Then a point $x$
in $X$ is almost periodic for $G$ if and only if $\overline{G(x)}$ is a minimal set for $G$.
\end{prop}

\begin{defn} \label{d52} \rm{We call the \textit{orbit closure relation} the set $R = \{(x, y)\in X\times X:
y\in \overline{G(x)}\}$. }
 \end{defn}

When $X$ is a circle, we have the following proposition.

\begin{prop} If $X$ is a circle, then ($G, X$) is pointwise almost periodic
if and only if either ($G, X$) is minimal or every non-endpoint of $X$ has finite orbit.
\end{prop}

\begin{thm}\label{t56} Let ($G, X$) be a group action, where $X$ is a local dendrite and $G$ is a group.
Then
\begin{enumerate}
 \item The following properties are equivalent:
\begin{itemize}
 \item[(i)] ($G, X$) is pointwise almost periodic;
 \item[(ii)] The orbit closure relation $R$ is closed.

 \end{itemize}
\item If $X$ is different from a circle, then the following are equivalent:
\begin{itemize}
 \item[(i)] ($G, X$) is pointwise almost periodic;
 \item[(ii)] Every non-endpoint of $X$ has finite orbit.
\end{itemize}
\end{enumerate}
 \end{thm}
\medskip

\begin{proof} (1) $(ii)\Rightarrow (i)$:  see  (\cite{Aus}, Proposition 1.1).
$(i)\Rightarrow (ii)$: Assume that ($G, X$) is pointwise almost periodic and let
$((x_n, y_n))_{n\in \mathbb{N}}\in X\times X$ such that $y_n\in \overline{G(x_n)}$ with
$\underset{n\to +\infty}\lim
d(x_n, x) = 0$, $\underset{n\to +\infty}\lim
d(y_n, y) =0$. Without loss of generality,
we assume that $\overline{G(x_n)}$ converges in Hausdorff dimension to a closed subset $M$ of $X$.
As $\overline{G(x_n)}$
is a minimal set for $G$, then by Theorem \ref{t610}, $M$ is a minimal set of $G$. Then we have
$\underset{n\to +\infty}\lim
d(x_n, M) =0$ and $\underset{n\to +\infty}\lim
d(y_n, M) =0$. It follows that $x, y\in M$ and therefore $M = \overline{G(x)}=\overline{G(y)}$. In particular,
$y\in \overline{G(x)}$ and so $R$ is closed. 

(2) $(ii)\Rightarrow (i)$: Assume that the orbit of every non-endpoint is finite. Since the set of non-endpoints of $X$ is dense in $X$, it follows by Corollary \ref{c6100},
that every endpoint of $X$ belongs
to a minimal set of $G$. Therefore assertion (i) follows.
$(i)\Rightarrow (ii)$:  Assume that ($G, X$) is pointwise almost periodic and let $x$ be a non-endpoint of $X$ that belong to $X\setminus Y$. Denote by $M = \overline{Gx}$.
Then $M\subset X\setminus Y$. By collapsing the graph $Y$ to a point $\gamma$, we get a
new dendrite $D$ and a new action on $D$ having $\gamma$ as a global fixed point. It follows that $M$ is either a finite orbit, or a Cantor set.
Let us prove that $M$ is finite.

\textit{Claim}. There exists a point $y\in D$ such that $Gy$ is finite and $x\in [y, a]$.

Indeed, there exists a point $z\in (D\backslash (M\cup E(X)))\cap C$, where $C$ is the connected component of $D\backslash \{x\}$ that does not contain $a$.
Set $N: = \overline{Gz}$. If $N$ is a finite orbit, the claim follows. Now assume that $N$ is a Cantor set, then by (\cite{MN1}, Theorem 5.4),
$N$ is the only infinite minimal subset in $[N]$, where
$[N]$ is the convex hull of $N$. So the orbit of every point in $[N]\backslash N$ is finite.
As $([N]\backslash N)\cap C\neq \emptyset$, the claim follows.

Now, by the claim, we have that $Gx\subset T$, where $T$ is a tree (since for any $g\in G$, $g(x)\in [a, g(y)]$ and $Gy$ is finite).
Therefore $M$ is a finite orbit by Theorem \ref{tr1}.

Now let us prove that the orbit of every point in $Y$ is finite.
Let $a$ be non-endpoint such that $a\in X\backslash Y$ and let $\{b\} = C\cup Y$, where $C$ is the connected component of
$X\setminus Y$ that contains $a$. Denote by $J = [a, b]$ the only arc in $X$ joining $a$ and $b$. Notice that the orbit of $Gb\subset Y$.
Thus, since $Ga$ is finite, on can deduce easily that $Z: = Y\cup \underset{g\in G}\cup g(I)$ is an invariant graph that contains
$Ga$. Therefore, by Theorem \ref{tr1}, all points in $Z$ (and in particular in $Y$) have finite orbits.
This proves the assertion (ii).
\end{proof}
\medskip


When $E(X)$ is countable, 
we have the following corollary.

\begin{cor}\label{c63} Let $($G, X$)$ be a group action, where $X$ is a local dendrite different from a
circle and $G$ is a group. Assume that $E(X)$ is countable. Then the following properties are equivalent:
\begin{itemize}
\item [(1)] ($G, X$) is pointwise almost periodic;
\item [(2)] ($G, X$) is pointwise periodic;
\end{itemize}
 \end{cor}

 \begin{proof} $(2)\Rightarrow (1)$: it is obvious. $(1)\Rightarrow (2)$: Assume that ($G, X$) is pointwise almost periodic.
 As in the proof of Theorem \ref{t56}, we collapse the graph $Y$ to a point $\gamma$ to get a dendrite $D$ with countable set
 of endpoints, and a new action on $D$ having $\gamma$ as a global fixed point and all other points are almost periodic.
 By Theorem \ref{tr33}, every point in $D$ has a finite orbit and so are points in $X\setminus Y$ under the action of $G$.
 Any point in the graph $Y$ is a non-endpoint, hence by Theorem \ref{t56}, all points in $X$ have finite orbits
 under the action of $G$.
 \end{proof}
\medskip

\textbf{Remark 2.} \begin{enumerate}
\item  Corollary \ref{c63} extends Theorem 1.2 in \cite{ha2}. Notice that if $X$ is
a graph different from a circle, Hattab \cite{ha2} proved precisely that
($G, X$) is pointwise almost periodic if and only if $G$ is finite.

\item If $X$ is a circle, it is clear that if $G$ is the group generated by an irrational rotation, then ($G, X$) is minimal
but not  pointwise periodic.

\item If $E(X)$ is uncountable, Corollary \ref{c63} is not true in general: In \cite{ef},
Efremova and Makhrova construct an example of a homeomorphism $f$ on a dendrite $X$ such that
$E(X)$ is a Cantor set, where every non-endpoint of $X$ is periodic and every point in $E(X)$ is almost periodic, but not periodic.
 \end{enumerate}
 \bigskip

We denote by $X/G$ the orbit space of this action
consisting of all of the $G$-orbits (i.e., the quotient space of $X$ by the equivalence relation
whose classes are the orbits of $G$). Since the orbit space $X/G$ is in general complicated, one can
study a natural simpler one, called \textit{the orbit class space} of $(X, G)$,
denoted by $X/\widetilde{G}$
and consisting of all of the orbit classes: two points of $X$ belong to the same orbit
class if the closures of their orbits are the same.

For an orbit $O$ of $G$, we call the \textit{class of }$O$ the union $\textrm{cl}(O)$ of all
orbits of $G$ having the same closure as $O$.
The two spaces are very linked together, since $X/G$ can be mapped onto $X/\widetilde{G}$
by the map $f$ which assigns to each $G$-orbit $O$ its class cl$(O)$. This map $f$ is an onto
quasi-homeomorphism (i.e. the map which assigns to each open set $V\subset X/\widetilde{G}$
the open set $f^{-1}(V)\subset X/G$ is a bijection, see \cite{gd}).

\begin{thm}\label{t56n} Let ($G, X$) be a group action, where $X$ is a local dendrite and $G$ is a group.
Then the following properties are equivalent:
\begin{itemize}
 \item[(1)] ($G, X$) is pointwise almost periodic;
  \item[(2)] $X/\widetilde{G}$ is Hausdorff.
\end{itemize}
 \end{thm}
\medskip

\begin{proof}  $(1)\Rightarrow (2)$: Assume that ($G, X$) is pointwise almost periodic. Then $R = \Gamma_{\widetilde{G}}$, where  $\Gamma_{\widetilde{G}}$
is the graph of the relation $\widetilde{G}$. By Theorem \ref{t56}, $\Gamma_{\widetilde{G}}$ is closed in $X\times X$.
We conclude that $X/\widetilde{G}$ is Hausdorff, since $X$ is compact. $(2)\Rightarrow (1)$: Assume that $X/\widetilde{G}$ is Hausdorff, then for every $G$-orbit $O$, its class
cl$(O)$ is closed in $X$. Hence cl$(O) = \overline{O}$ since $O\subset cl(O)\subset \overline{O}$. It follows that $\overline{O}$ is a minimal set. Therefore assertion (1) follows.
\end{proof}
\medskip

For a single homeomorphism, Jmel \cite{j} (resp. Hattab and Salhi
\cite{ha2}) has shown that if $f$ is a pointwise periodic homeomorphism of $X$ which is a dendrite (resp. a graph), then the orbit space $X/f$ is Hausdorff.
 We extend these results to a group action on a local dendrite as follows:
\medskip

\begin{thm}\label{t59} Let $($G, X$)$ be a group action, where $X$ is a local dendrite and $G$ is a
group. Then we have the following properties:
\begin{enumerate}
 \item  If ($G, X$) is pointwise periodic, then $X/G$ is Hausdorff.

 \item If $X/G$ is Hausdorff, then every non-endpoint of $X$ has finite orbit or $X$ is a circle which is a $G$-orbit.
\end{enumerate}
\end{thm}

\begin{proof} Assertion $(1)$: Assume that ($G, X$) is pointwise periodic,
then $X/G = X/\widetilde{G}$. Hence by Theorem \ref{t56n},
 $X/G$ is Hausdorff. Assertion (2): Assume that  $X/G$ is Hausdorff. Then all $G$-orbits are closed in $X$ and hence minimal sets.
 Assume that $X$ is a circle. By Corollary \ref{c33}, each orbit is finite or each orbit is $X$. If $X$ is an interval,
 then it is easy to see that every point has finite orbit. So suppose that $X$ is neither circle nor interval,
 then its set of branch points $B(X)$ of $X$ is nonempty. Let $x\in B(X)$. Then $Gx\subset B(X)$ by Lemma \ref{g1}.
 Since $B(X)$ is countable, $Gx$ is finite. If $X$ is a dendrite, then by (\cite{MN1}, Theorem 5.4), every
 non-endpoint of $X$ has finite orbit. Now assume that $X$ is a local dendrite not a dendrite (that is $Y$ is non empty), then by collapsing $Y$ to a point $\gamma$, we get a dendrite $D$ and a new action having $\gamma$ as a global fixed point and such that any other point has minimal orbit. Hence, every non-endpoint in $D$ has finite orbit and so the same holds under the action of $G$ on $X$ for any point $x$ outside the graph $Y$. Since $B(X)\cap Y\neq\
\emptyset$, there is a point in $Y$ with finite orbit, it follows that the same holds for any point in $Y$.
 \end{proof}
\medskip

\begin{cor}\label{c58} Let $($G, X$)$ be a group action, where $X$ is a local dendrite and $G$ is a countable
group. Then the following properties are equivalent:
\begin{itemize}
 \item [(1)] ($G, X$) is pointwise periodic;
\item [(2)] $X/G$ is Hausdorff.
\end{itemize}
\end{cor}

\begin{proof} $(1)\Rightarrow (2)$: it is already done by Theorem \ref{t59}. $(2)\Rightarrow (1)$: Assume that $X/G$ is Hausdorff. Then all $G$-orbits are closed in $X$ (hence minimal sets) and
so $X/G = X/\widetilde{G}$. Since $G$ is countable, so by Lemma \ref{l21}, all $G$-orbits are finite;
equivalently ($G, X$) is pointwise periodic.
\end{proof}
\bigskip


\textit{Acknowledgements.} This work was supported by the research unit:
``Dynamical systems and their applications'' (UR17ES21),  of Higher Education and Scientific Research,
Tunisia.

\bibliographystyle{amsplain}

\end{document}